\documentclass[11pt]{amsart}
\pdfoutput=1
\usepackage[utf8]{inputenc}
\usepackage{amsmath, amssymb}
\usepackage[english]{babel}
\usepackage{graphicx}
\usepackage{amsthm}
\usepackage{array}
\usepackage{amsthm}
\usepackage{mathtools}
\usepackage{thmtools}
\usepackage[margin=.85 in]{geometry}
\allowdisplaybreaks

\usepackage[dvipsnames]{xcolor}
\usepackage{hyperref}
\definecolor{candyapplered}{rgb}{1.0, 0.03, 0.0}
\definecolor{mediumblue}{rgb}{0.0, 0.0, 0.8}
\hypersetup{
pdfauthor={Tarik Aougab and Jonah Gaster},
pdftitle={k-systems on the torus},
bookmarksnumbered,
colorlinks=true,
linkcolor=mediumblue,
citecolor=candyapplered,
urlcolor=blue}

\usepackage{caption,subcaption}
\usepackage[ocgcolorlinks]{ocgx2}

\declaretheorem[numberwithin=section]{theorem}
\declaretheorem[sibling=theorem, style=definition]{definition}
\declaretheorem[sibling=theorem]{lemma}
\declaretheorem[sibling=theorem, style=remark]{remark}
\declaretheorem[sibling=theorem]{proposition}

\numberwithin{equation}{section}

\newcommand{\bZ}{\mathbb{Z}}

\newcommand{\bN}{\mathbb{N}}
\newcommand{\bH}{\mathbb{H}}

\newcommand{\bR}{\mathbb{R}}
\newcommand{\bQ}{\mathbb{Q}}

\newcommand{\cA}{\mathcal{A}}

\newcommand{\PSL}{\mathrm{PSL}}
\newcommand{\cF}{\mathcal{F}}
\newcommand{\cH}{\mathcal{H}}

\newcommand{\ch}{\mathrm{ch}}
\newcommand{\ach}{\mathrm{ach}}
\newcommand{\reg}{\mathrm{reg}}
\newcommand{\Far}{\mathrm{Far}}

\mathtoolsset{showonlyrefs=true}

\begin{document}
\title[curves on the torus intersecting at most $k$ times]
{Curves on the torus intersecting at most $k$ times}
\author[Aougab, Gaster]{Tarik Aougab and Jonah Gaster}
\date{August 18, 2020}

\address{Department of Mathematical Sciences, University of Wisconsin-Milwaukee}
\email{gaster@uwm.edu}

\address{Department of Mathematics, Haverford College}
\email{taougab@haverford.edu}


\begin{abstract}

We show that any set of distinct homotopy classes of simple closed curves on the torus that pairwise intersect at most $k$ times has size $k+O(\sqrt k \log k)$. 
Prior to this work, a lemma of Agol, together with the state of the art bounds for the size of prime gaps, implied the error term $O(k^{21/40})$, and in fact the assumption of the Riemann hypothesis improved this error term to the one we obtain $O(\sqrt k\log k)$.
By contrast, our methods are elementary, combinatorial, and geometric.

\end{abstract}

\maketitle

\section{Introduction}

Let $T \approx \bR^2/\bZ^2$ be the closed oriented surface of genus one.
We indicate the homotopy class of an embedding of $S^1$ briefly by `curve'. 
By pulling a curve tight and lifting it to the universal cover, the collection of curves on $T$ is in one-to-one correspondence with slopes $\bQ \cup \{\infty\}$. 
From this vantage point, the \emph{intersection number} of a pair of curves on $T$ (that is, the minimum possible number of intersection points among representatives from the pair of homotopy classes) can be computed explicitly via
\[
\iota\left( \frac pq , \frac ab\right) = \ | \ pb-qa \ | \ ~.
\]
A collection of curves is called a \emph{$k$-system} when any pair of curves has intersection number at most $k$.

Let $\eta_S(k)$ equal the maximum size of a $k$-system on the closed surface $S$.
It was first shown by \cite{JMM} that $\eta_S(k)$ goes to infinity with $k$. 
The determination of the growth rate of $\eta_S(k)$, as a function of both $k$ and the genus $g$ of $S$, is a subtle counting problem, about which much remains unknown \cite{Przytycki, Aougab, ABG, Greene1, Greene2}. 
Notably, Greene has used probabilistic methods, leveraging the hyperbolic geometric bounds of Przytycki, to obtain $\eta_S(k) =O( g^{k+1}\log g)$, when $k$ is fixed and $g$ grows \cite[Thm.~3]{Greene2}.

For the study of $\eta_S(k)$ with $g$ fixed, the simplest nontrivial case is evidently $S=T$.
While studying Dehn filling slopes of 3-manifolds, Agol observed that $\eta_T(k)$ is at most one more than the smallest prime greater than $k$ \cite{Agol}, and via the Prime Number Theorem this implies $\frac{\eta_T(k)}k \to 1$ as $k\to \infty$ \cite{Agol-unpub}.

More can be said. 
The size of prime gaps, large and small, is a major field of study.
The currently best upper bound is due to Baker-Harman-Pintz, which, together with Agol's observation, implies that $\eta_T(k) = k+O(k^{21/40})$ \cite{BHP}. 
Cram\'er showed that a positive resolution of the Riemann hypothesis would provide $\eta_T(k) = k + O(\sqrt k \log k)$ \cite{Cramer1}, and he formulated a stronger conjecture that would imply $\eta_T(k) = k+O((\log k)^2)$ \cite{Cramer2}; although there seems to be general suspicion in the analytic number theory community that Cram\'er's error term should be replaced by $O((\log k)^{2+\epsilon})$ \cite{Granville, OSS}.

All of these estimates pass through Agol's remarkable prime number bound, but it is reasonable to be skeptical about whether estimation of $\eta_T(k)$ should depend on such notoriously subtle and difficult questions.
The purpose of this note is to sharpen currently available estimates, without reference to fine data about the distribution of the primes.
Our methods are elementary, combinatorial, and hyperbolic.

\begin{theorem}
\label{main thm}
There is a constant $C>0$ so that $\eta_T(k)\le k+C\sqrt k \log k$.   
\end{theorem}

As for sharpness, it deserves remarking that there is a dearth of nontrivial lower bounds for $\eta_T(k)$.
In fact, we are unaware of any example of a $k$-system of size $k+7$ on the torus.

The function $\eta_T(k)$ admits a dual formulation: let $\kappa_S(n)$ indicate the minimum, taken over collections of $n$ curves on $S$, of the maximum pairwise intersection. 
(For clarity, we will often use `$n$' to indicate the size of a set of curves and `$k$' for an intersection number.)
It is not hard to see that
\[
\eta_S(k) = \max \{ n: \kappa_S(n) \le k\} \ , \text{ and } \
\kappa_S(n) = \min \{k : \eta_S(k)\ge n\}~.
\]
Our path towards \autoref{main thm} will be to first estimate $\kappa_T(n)$ from below.

There is a kind of convexity to exploit in the study of $\kappa_T$, originally observed by Agol.
As remarked above, curves on $T$ are in correspondence with slopes $\bQ\cup \{\infty\}$. 
The latter form the vertices of the \emph{Farey complex} $\cF$, in which a set of slopes form a simplex when they pairwise intersect once.  
Any collection of $n$ curves that is maximal with respect to inclusion among $k$-systems determines a collection of vertices in $\cF$ so that the induced simplicial complex is a triangulated $n$-gon (see \autoref{lem: convexity} for detail).

Conversely, any triangulation of an $n$-gon can be realized as a subcomplex of $\cF$, in a way that is unique up to the action of $\PSL(2,\bZ) \curvearrowright \cF$ by simplicial automorphism. As $\PSL(2,\bZ)$ preserves the intersection form, the (multi-)set of pairwise intersection numbers is a well-defined function on the set of triangulations of an $n$-gon. 
The set of triangulations of an $n$-gon forms the vertex set of a well-studied simplicial complex in combinatorics called the \emph{associahedron} $\cA_n$
(there is a slight indexing issue; the object we refer to as $\cA_n$ is the $(n-2)$-dimensional associahedron).
One obtains a `max intersection' function $\kappa: \cA_n \to \bN$ induced by the intersection form on $\cF$,
and the above discussion leads to $\kappa_T(n)=\min \kappa$ (see \autoref{prop:kappa=kappa}). 

\autoref{main thm} follows from the following:

\begin{theorem}
\label{thm: kappa}
There is a constant $C>0$ so that, for any $\tau\in \cA_n$, we have $\kappa(\tau) \ge n - C \sqrt n \log n $.
\end{theorem}

We briefly describe the proof of this theorem.
The Farey complex $\cF$ admits a natural embedding into a compactification of the hyperbolic plane $\bH^2 \,\cup\, \partial_\infty \bH^2$, so that the vertices of $\cF$ embed naturally as $\bQ \cup\{\infty\} \hookrightarrow \bR\cup\{\infty\} \approx \partial_\infty \bH^2$, with edges between vertices mapping to geodesics.
The hyperbolic plane $\bH^2$ admits a maximal $\PSL(2,\bZ)$-invariant horospherical packing $\{H_{p/q}: \frac pq \in \bQ\cup\{\infty\}\}$, where $H_{p/q}$ is centered at $p/q\in \bR\cup\{\infty\}$, so that a set of slopes span a simplex in $\cF$ precisely when the corresponding horospheres are pairwise tangent. (The horospheres $\{H_{p/q}\}$ are called \emph{Ford circles} in the literature \cite{Ford, ConwayGuy, BonahonVid}.)

A sketch of our proof of \autoref{thm: kappa} is as follows:
\begin{enumerate}
\item Locate a `nice horoball' $H$ for $\tau$, so that $\mathrm{ht}(\tau,H)$, the \emph{height of $\tau$ relative to $H$}, is $O\left(\sqrt{\kappa(\tau)}\right)$. 
See \autoref{def:horoball height width} and \autoref{prop: exists horoball}.

\item Use $H$ to construct a convex combination of pairwise intersection numbers for $\tau$ whose sum is at least $n - O(h \log h)$, where $h= \mathrm{ht}(\tau,H)$. 
It follows that there is a pair of horoballs of $\tau$ with intersection number at least $n-O(h \log h)$.
See \autoref{prop: control with height}.

\end{enumerate}
The proof of \autoref{thm: kappa} is now one line: if $\kappa(\tau)\le n$, then $\kappa(\tau) \ge n-C \sqrt n \log n$.

The first step above uses the hyperbolic geometry of $\bH^2$ in an essential way, in which we exploit a simple relationship between intersection numbers, hyperbolic geometry, and Ford circles (see \autoref{lem:i vs d}).

\subsection*{Organization}
We describe the reduction from $\kappa_T(n)$ to $\kappa:\cA_n\to\bN$ in \S\ref{sec:prelim}, analyze several examples in \S\ref{sec:examples}, bound $\kappa(\tau)$ from below in \S\ref{sec:estimating kappa}, locate a good horoball for $\tau$ in \S\ref{sec:controlling height}, and prove \autoref{main thm} in \S\ref{sec:pf of main thm}.

\subsection*{Acknowledgements}
The authors thank Josh Greene and Ian Agol for valuable feedback on an early draft of this paper. We are especially grateful to Ian Agol for sharing with us an unpublished note that contained \autoref{lem: convexity}, and for suggesting an alternative to our proof of \autoref{prop: exists horoball} (see \autoref{rem: Agol proof}).

\section{Preliminaries}
\label{sec:prelim}

We collect here some useful facts about intersection numbers and the Farey graph $\cF$. 
For more of the beautiful connections between hyperbolic geometry, the Farey graph, continued fractions, and Diophantine approximation, we suggest the reader consult \cite{Hatcher, Series, Series2, Springborn}.

\subsection{Horoballs in trees, Farey labellings, and intersection numbers}
Dual to a triangulation of an $n$-gon $\tau\in\cA_n$ there is a trivalent tree with $n$ leaves embedded in the plane, which we refer to as $\tau^*$.
Because $\tau^*$ is embedded in the plane, the three edges incident to vertices of $\tau^*$ are cyclically ordered. 
Hence any non-backtracking path in $\tau^*$ induces a sequence of left-right turns.

The vertices of $\tau$ (that is, the slopes of the $k$-system) correspond to `horoball' regions in $\tau^*$: 

\begin{definition}[`Horoballs in trees']
\label{def:horoball}
A \emph{horoball} of $\tau$ is a union of edges in a path of the dual tree $\tau^*$ that is composed of uni-directional turns (that is, only left or only right turns), which is moreover maximal with respect to inclusion among all such uni-directional subsets of $\tau^*$.
\end{definition}

For any triangulation $\tau$, the dual tree $\tau^*$ admits an orientation-preserving embedding to the regular trivalent tree dual to $\cF$.
Any such choice of an embedding determines a map from horoballs of $\tau^*$ to $\bQ\cup \{\infty\}$, by recording the center of the corresponding horoball in $\bH^2$.

\begin{definition}[`Farey labellings and intersection numbers']
A \emph{Farey labelling} of $\tau$ is the map from horoballs to $\bQ\cup\{\infty\}$ obtained from an orientation-preserving embedding from $\tau^*$ to the tree dual to $\cF$.
The \emph{intersection number} $\iota(H_1,H_2)$ of a pair of horoballs $H_1$ and $H_2$ is given by the intersection number of the slopes corresponding to $H_1$ and $H_2$ in a Farey labelling of $\tau$.
\end{definition}
\noindent
We leave it as an exercise for the reader to show that intersection numbers in $\tau$ are well-defined.

Farey labellings are especially pleasant because the vertices spanning a simplex of $\cF$ satisfy a remarkably simple relationship. Namely, if $\frac pq$ and $\frac ab$ span an edge of $\cF$, then the two other vertices of $\cF$ that span a triangle with $\frac pq$ and $\frac ab$ are $\frac{p+ a}{q+ b}$ and $\frac{p-a}{q- b}$; this is the `Farey addition' rule. 
Farey addition can be used to construct a Farey labelling of $\tau$: 
Choose labels $1/0$, $0/1$, and $1/1$ for the three horoballs incident to some vertex of $\tau^*$, and use Farey addition to successively add labels to neighboring horoballs.

\subsection{Monotonicity of intersection numbers and left-right sequences}
\label{subsec:monotonicity}
The intersection number $\iota(H_1,H_2)$ admits a description more intrinsic to the structure of $\tau^*$, which we now describe.
There is a unique (possibly degenerate) non-backtracking path $\sigma$ between the pair of horoballs $H_1$ and $H_2$, and this path determines a sequence of left-right turns $(\ell_1,\ell_2,\ldots,\ell_s)$, where $\sigma$ makes $\ell_1$ turns in the same direction, followed by $\ell_2$ turns in the opposite direction, etc.
The quantity $\iota(H_1,H_2)$ is given by the numerator of the continued fraction with coefficients $(\ell_1,\ell_2,\ldots,\ell_s)$ \cite[Thm.~5.3]{GLRRX}.

\begin{remark}
\label{rem:ambiguity}
Observe that there is ambiguity in this computation of $\iota(H_1,H_2)$.
For one, the non-backtracking path $\sigma$ may go either from $H_1$ to $H_2$, or from $H_2$ to $H_1$.
Moreover, one must declare that $\sigma$ is starting with either `left' or `right' at its origin vertex, so that it can be observed whether $\sigma$ is switching directions or not at later vertices.
These choices may be made arbitrarily and independently, and this ambiguity has no affect on the calculation of $\iota(H_1,H_2)$. See \cite[Fig.~3, Ex.~1]{GLRRX}. 
\end{remark}

This viewpoint suggests a certain monotonicity.

\begin{lemma}[`Monotonicity of intersection numbers']
\label{lem:monotonicity}
Suppose that $\sigma$ and $\sigma'$ are non-backtracking paths with respective left-right sequences $(\ell_1,\ldots,\ell_s)$ and $(\ell_1',\ldots,\ell_{s'}')$. 
If $s'\ge s$ and $\ell_i'\ge \ell_i$ for each $i=1,\ldots,s$, then the intersection number determined by $\sigma'$ is at least that determined by $\sigma$.
\end{lemma}

\begin{proof}
This lemma is almost exactly \cite[Lem.~5.5]{GLRRX}, with the sole difference that we may have $s'>s$.
Therefore to prove the claim we may assume that $\sigma'$ contains $\sigma$ as an initial subpath.

Choose a Farey labelling with label $1/0$ at the horoball forming the origin of $\sigma$ (and $\sigma'$), and labels $0/1$ and $1/1$ at the two neighboring horoballs that intersect $\sigma$. 
Compare the denominators of the Farey labels of the horoballs at the terminuses of $\sigma$ and $\sigma'$; because these labels are computed using Farey addition, it is evident that the denominator of the horoball for $\sigma'$ is at least that of $\sigma$. 
The intersection number of any horoball with $1/0$ is given by the denominator of its Farey label, so the claim follows.
\end{proof}

\subsection{Intersection numbers and hyperbolic distance}
The quotient of $\bH^2$ by $\PSL(2,\bZ)$ is a hyperbolic orbifold with one cusp and two orbifold points, one of order $2$ and one of order $3$. 
The preimage of the maximal horoball neighborhood of the cusp under the covering projection $\bH^2 \to \bH^2/\PSL(2,\bZ)$ is $\cH=\{H_{p/q}:p/q\in\bQ\cup\{\infty\}\}$, 
a $\PSL(2,\bZ)$-invariant collection of horoballs centered at the completed rationals. 
The following lemma is an exercise in hyperbolic geometry.

\begin{lemma}
\label{lem:distance from horoball}
Every point in $\bH^2$ is within $\log\frac2{\sqrt 3}$ of a horoball in $\cH$.
\end{lemma}

There is a simple fundamental relationship between intersection numbers of curves on the torus and hyperbolic distance between the corresponding horoballs.

\begin{lemma}
\label{lem:i vs d}
We have 
$\ 
\displaystyle d_{\bH^2}\left(H_{p/q},H_{a/b} \right) = 
2 \log \iota\left( \frac pq , \frac ab \right) ~$
for any $\frac pq, \frac ab \in \cF$.   
\end{lemma}

\begin{proof}
Applying an element of $\PSL(2,\bZ)$, we may assume that $p/q=\infty$ in the upper half-plane model for $\bH^2 \, \cup \, \partial_\infty \bH^2$. 
The horosphere $H_{a/b}$ is given by $\{z: \left| z-(\frac1b + \frac i{2b^2}) \right| = \frac1{2b^2}\}$ (see e.g.~\cite{Athreya}), so 
\[
d_{\bH^2}\left(H_{p/q},H_{a/b} \right) =
d_{\bH^2}\left( a/b + \frac i{b^2} , a/b + i\right) = 2\log b = 2 \log \iota\left( \frac pq , \frac ab \right) ~. \qedhere
\]
\end{proof}

\subsection{Width and height for horoballs}\label{sec:width/height}
The interior of each edge of $\tau^*$ is incident to exactly two horoball regions. 
Thus, for any choice of horoball $H$ in $\tau^*$, there are exactly two other horoball regions, distinct from $H$, that are incident to the interiors of the extreme edges of $H$. Call these $H_1$ and $H_2$.

\begin{definition}
\label{def:horoball height width}
The \emph{width of $\tau$ relative to $H$} is $w = \iota(H_1,H_2)$. The \emph{height of $\tau$ relative to $H$} is 
\[
\mathrm{ht}(\tau,H) = 
\max\, \{ \ \iota(H, H') \ : H' \text{ is a horoball of }\tau \ \}~.
\]
\end{definition}

For the remainder of this article, we will suppress the difference between the triangulation $\tau\in\cA_n$ and its dual tree $\tau^*$. 
The translation between them is quite natural, and the difference can henceforth be understood from context.

\subsection{The two kappas}
Recall from the introduction the quantity $\kappa_T(n)$, which is the minimum, taken over collections $\Gamma$ of $n$ curves on $S$, of the maximum pairwise intersection number of curves in $\Gamma$.
\begin{definition}[`Max Intersection Function']
\label{def:kappa}
The function $\kappa:\cA_n\to \bN$ is defined by
\[
\kappa(\tau) = \max\{\ \iota(H_1,H_2) : H_1 \text{ and }H_2 \text{ are horoballs of }\tau \ \}~.
\]
\end{definition}
\noindent As noted in the introduction, we claim that $\kappa_T(n)= \min \{ \kappa(\tau): \tau\in \cA_n \}$.

That $\kappa_T(n)\le \min \kappa$ is easy: for any $\tau\in\cA_n$, choose a Farey labelling. 
The quantity $\kappa(\tau)$ is equal to the maximum pairwise intersection number of the $n$ slopes obtained in this Farey labelling, and $\kappa_T(n)$ is the minimum of the maximum pairwise intersection of any $n$ slopes, so $\kappa_T(n) \le \kappa(\tau)$ for each $\tau\in\cA_n$. 

The reverse inequality is slightly less obvious, and relies on a certain convexity of maximal $k$-systems in $\partial_\infty \bH^2$. 
The following lemma makes this precise.\footnote{We are grateful to Ian Agol for sharing an unpublished note with us which contained \autoref{lem: convexity} \cite{Agol-unpub}.}

\begin{lemma}
\label{lem: convexity}
If $\Gamma$ is a $k$-system on $T$ which is maximal with respect to inclusion among $k$-systems, then $\cF$ induces a triangulation of the $n$-gon which forms the convex hull of $\Gamma\subset \partial_\infty \bH^2$.
\end{lemma}

\begin{proof}
Let $g_{ab}$ indicate the geodesic in $\bH^2$ with endpoints $a,b\in\partial_\infty \bH^2$. 
Suppose that $\alpha,\beta\in\Gamma\subset \cF$, that $\Delta$ is a Farey triangle intersecting $g_{\alpha\beta}$, and that $\delta$ is a vertex of $\Delta$ (and, hence, of $\cF$). 
\autoref{lem:monotonicity} implies that both $\iota(\delta,\alpha)$ and $\iota(\delta,\beta)$ are at most $\iota(\alpha,\beta)$, which is at most $k$ by assumption.
For any $\gamma\in\Gamma$, either $g_{\alpha\gamma}$ or $g_{\beta\gamma}$ intersect $\Delta$, so it follows that $\iota(\gamma,\delta)\le k$ as well. 
Maximality of $\Gamma$ implies that $\delta\in\Gamma$, so the convex hull of $\Gamma$ is equal to the union of Farey triangles spanning elements of $\Gamma$.
\end{proof}

This demonstrates that $\kappa_T(n) \ge \min \kappa$, so we may conclude:

\begin{proposition}
\label{prop:kappa=kappa}
We have $\displaystyle \kappa_T(n) = \min_{\tau\in\cA_n} \kappa$.
\end{proposition}

\section{Illustrative examples}
\label{sec:examples}

There are several natural elements of $\cA_n$ that we can use to observe $\kappa:\cA_n\to \bN$.

\begin{remark}
\label{rem:unlabelled}
Technically, the vertices of the associahedron correspond to triangulations of a \emph{labelled} convex polygon.
Notice however that the max intersection function $\kappa:\cA_n\to\bN$ is invariant under permutation of labels, so we can safely refer to $\kappa(\tau)$ for elements $\tau$ of $\cA_n$ without reference to a particular ordering of the horoballs of $\tau$.
\end{remark}

\begin{itemize}

\item The element $\ch(n)\in\cA_n$ (for `chain') contains a horoball of width $n$. We have $\kappa(\ch(n))=n$, and the height relative to the horoball of width $n$ is $1$.
\\

\item The element $\ach(n)\in\cA_n$ (for `alternating chain') contains a path of length $n-3$ that switches direction $n-4$ times. Here we have $\kappa(\ach(n))=F_n$, the $n$th Fibonacci number, the largest width horoball of $\ach(n)$ is $3$, and the height relative to any horoball is at least $F_{\lfloor \frac n2\rfloor }$.
\\

\item The element $\reg(r)\in\cA_n$ (for `regular'), with $n=3\cdot 2^{r-1}$, is formed by choosing the subtree of the homogeneous (infinite) trivalent tree that is induced on all vertices at combinatorial distance at most $r$ from a fixed vertex. The tree $\reg(r)$ contains the alternating chain $\ach(2r+1)$ as a subtree, and in fact we have $\kappa(\reg(r))=\kappa(\ach(2r+1))=F_{2r+1}$. 
The largest width of a horoball of $\reg(r)$ is given by $2r-1$, and the height of $\reg(r)$ relative to this horoball is $F_{r+1}$.
\\

\item The element $\Far(h)\in\cA_n$ (for the `Farey series'), with $n=2+\sum_{k\le h} \phi(k)$, is the subgraph of $\cF$ induced on fractions in $\bQ\cap [0,1]$ that can be written with denominator $\le h$, together with $1/0$. 
Observe that $\kappa(\Far(h)) =h^2-2h$, the largest width of a horoball of $\Far(h)$ is given by $h$, while the height relative to this horoball is given by $h-1$. \\

\end{itemize}

We collect this information in \autoref{pic:elements} and \autoref{table}.

\begin{figure}
\centering
\vspace{1.5cm}
\begin{minipage}[t][3cm][t]{.2\textwidth}
\centering
\vspace{-.7cm}
\includegraphics[width=4cm]{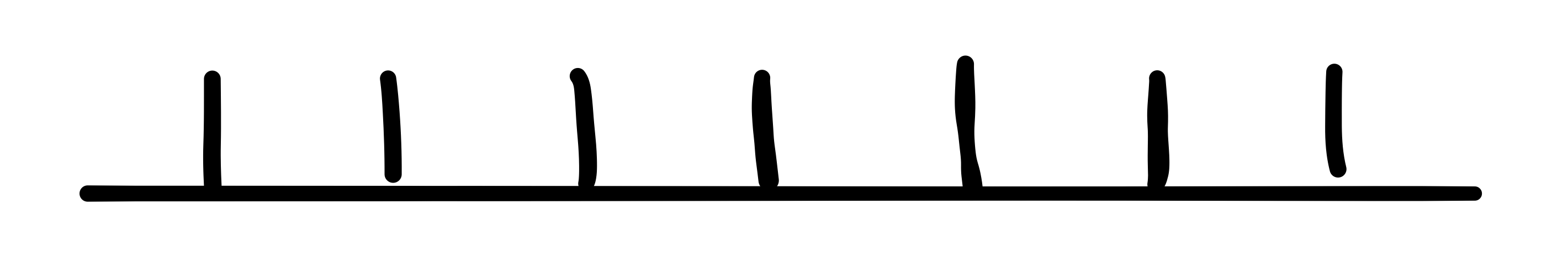}
\vspace{1.5cm}
\subcaption{$\ch(n)$}
\label{pic:chain}
\end{minipage}\hfill
\begin{minipage}[t][3cm][t]{.2\textwidth}
\centering
\includegraphics[width=4cm]{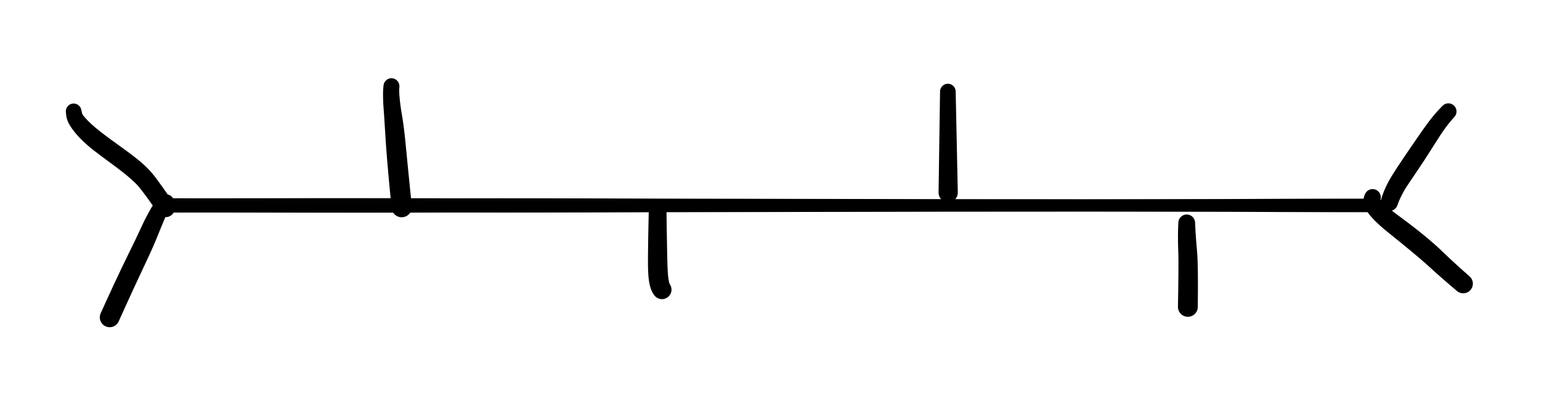}
\vspace{1.5cm}
\subcaption{$\ach(n)$}
\label{pic:achain}
\end{minipage}\hfill
\begin{minipage}[t][3cm][t]{.2\textwidth}
\centering
\vspace{-2cm}
\includegraphics[width=4cm]{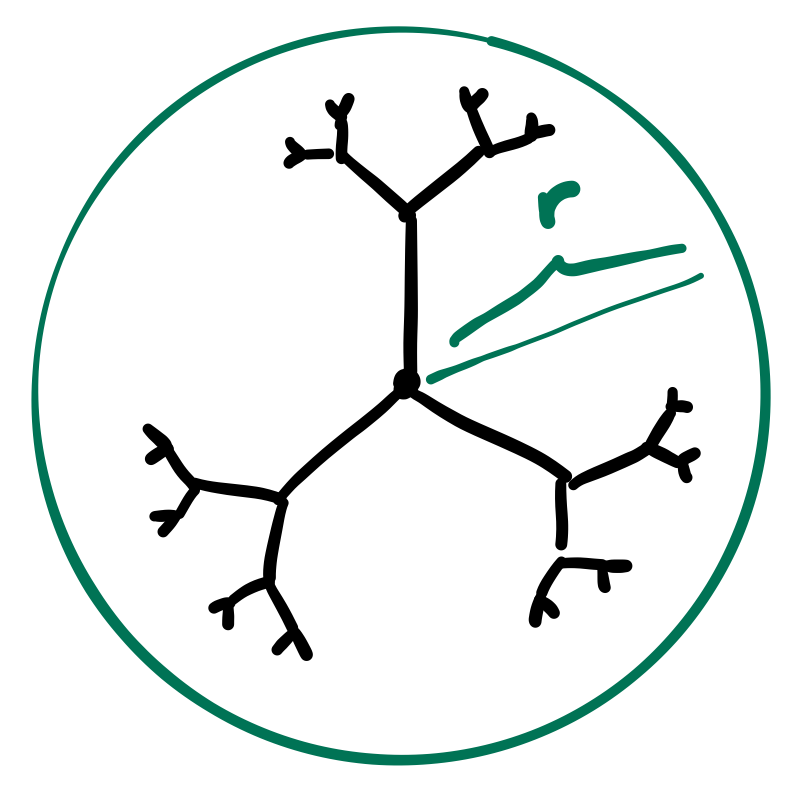}
\subcaption{$\reg(r)$}
\label{pic:chain}
\end{minipage}\hfill
\begin{minipage}[t][3cm][t]{.2\textwidth}
\vspace{-1.4cm}
\centering
\includegraphics[width=4cm]{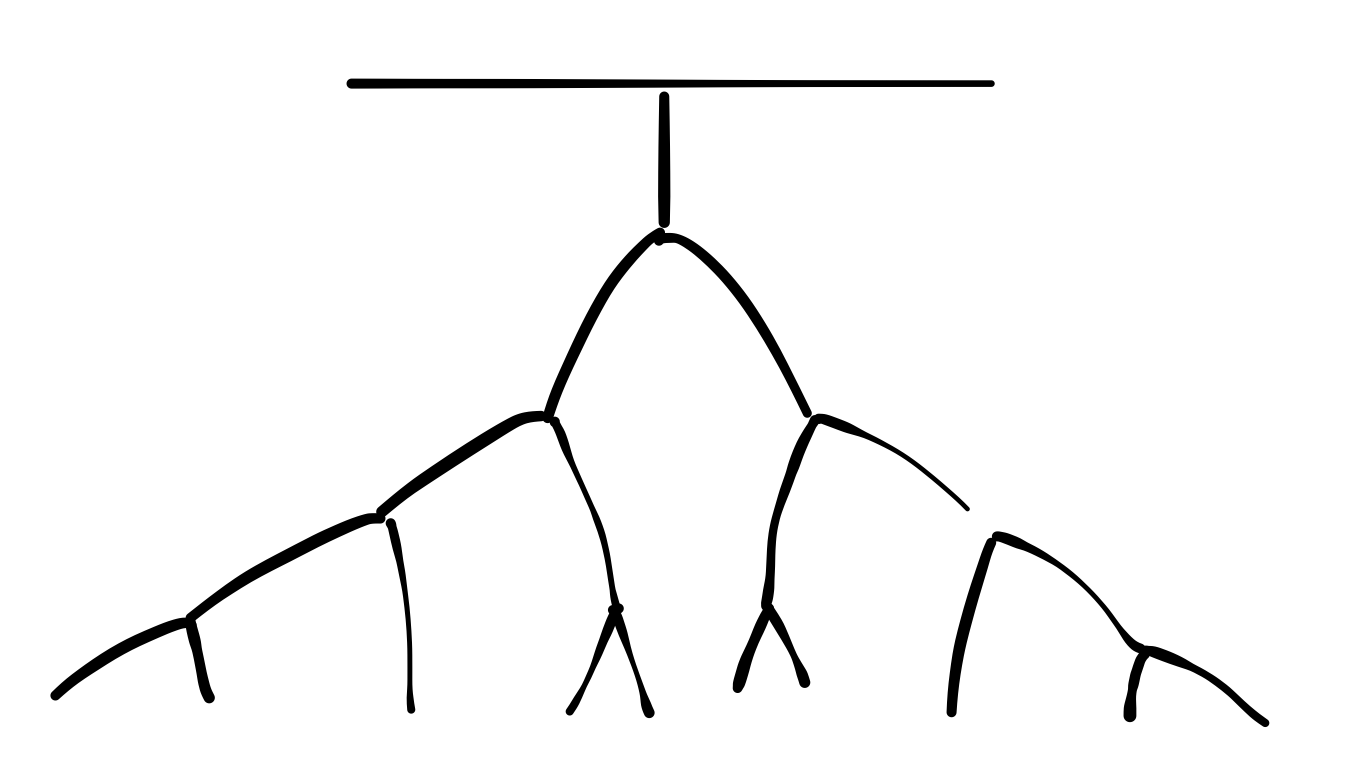}
\vspace{.6cm}
\subcaption{$\Far(h)$}
\label{pic:chain}
\end{minipage}\hfill
\caption{Several dual trees of elements in $\cA_n$.}
\label{pic:elements}
\end{figure}

\renewcommand{\arraystretch}{1.5}
\begin{table}[]
\begin{tabular}
{|
>{\centering}p{.1\textwidth} |
>{\centering}p{0.2\textwidth} |
>{\centering}p{0.3\textwidth} |
>{\centering\arraybackslash}p{0.2\textwidth} |
}
\hline
$\tau$ & $\kappa(\tau)$ & Largest width horoball $H$ & $\mathrm{ht}(\tau,H)$ \\
\hline 
\hline
$\ch(n)$ & $n$ & $n-2$ & $1$  \\
\hline 
$\ach(n)$ & $ \Phi^n$ & $3$ & $\ge \Phi^{n/2}$  \\
\hline  
$\reg(r)$ & $ n^{2\log_2\Phi}$ & $2\log_2n $ & $ n^{\log_2\Phi}$  \\
\hline  
$\Far(h)$ & $ \frac{\pi^2}3 n $ & $\sqrt n$  &  $\sqrt n$ \\
\hline
\end{tabular}
\vspace{.5cm}
\caption{Some data for attractive elements of $\cA_n$.  
Some entries include only leading-order terms, ignoring multiplicative constants. 
Note that $\Phi = \frac{1+\sqrt 5}2$.
}
\label{table} 
\end{table}

\begin{remark}
\label{rem:An diameter}
Observe the large difference between $\kappa(\ach(n))\approx (\frac{1+\sqrt5}2)^n$ and $\kappa(\ch(n))=n$. 
Though we have not discussed it, there is a natural simplicial structure on $\cA_n$, with edges between triangulations of an $n$-gon that differ by a single diagonal flip. 
It is not hard to see that the diameter of $\cA_n$ is at most $2n$ (in fact, this quantity can be determined precisely \cite{STT, Pournin}), so it follows that the change in $\kappa$ across an edge of $\cA_n$ can be arbitrarily large. 
\end{remark}

\section{Estimating kappa using heights}
\label{sec:estimating kappa}

Let $\tau\in\cA_n$. 
The strategy to obtain a lower bound for $\kappa(\tau)$ is to find a set of pairwise intersection functions $\{I_\alpha\}$ for $\tau$, and estimate the convex combination
\begin{equation}
\label{eq:goal}
\sum_\alpha r_\alpha I_\alpha \ge n - \epsilon(n)~,
\end{equation}
for some set of non-negative weights $\{r_\alpha\}$ with $\sum r_\alpha =1$, and error term $\epsilon(n)$.
Of course, we have $\kappa(\tau)\ge I_\alpha$ for all $\alpha$, and by convexity we must have $I_\alpha \ge n-\epsilon(n)$ for some $\alpha$.

We will make use of three facts from classical analytic number theory, which we group together in a single lemma for convenience.
Below we indicate the interval $\{1,\ldots,m\}$ by $[m]$ and the subset of $\{1,\ldots,m\}$ relatively prime to $s$ by $[m]_s$, e.g.~Euler's totient function is $\phi(m) = \#[m]_m.$ 
The number of divisors of $n$ is indicated by $d(n)$.

\begin{lemma}
\label{lem:number theory}
We have the following estimates:
\begin{align}
\label{eq:euler count}
\#[m]_n & = \frac mn \phi(n) + O\left(d(n)\right)~, \\
\label{eq:dirichlet}
\sum_{k\le h} d(k)& = h \log h + O(h)~, \text{ and }\\
\label{eq:walfisz}
\sum_{k\le h} \frac {\phi(k)}k &= O(h)~.
\end{align}
\end{lemma}

The estimate \eqref{eq:euler count} is a standard application of M\"obius inversion \cite[Lem.~3.4]{Cohen}. 
The second estimate \eqref{eq:dirichlet} is a weaker version of a famous theorem of Dirichlet \cite[Ch.~3]{Apostol}, and a more precise form of \eqref{eq:walfisz} can be found in \cite{Walfisz}.
(Note that the error term in \eqref{eq:euler count} is in fact $O(\vartheta(n))$, where $\vartheta(n)$ is the number of square-free divisors of $n$. 
However, in the sum $\sum_{j\le h} \vartheta(j)$, one finds the same order of growth as $\sum_{j\le h} d(j)$ \cite{Mertens}, so in our application, \autoref{lem:count n}, this improvement is immaterial.)

In this section we will show:

\begin{proposition}
\label{prop: control with height}
Let $H$ be a horoball of $\tau\in\cA_n$, and let $h = \mathrm{ht}(\tau,H)$. 
There is a constant $C>0$ so that we have $\kappa(\tau) \ge n - C h \log h$.
\end{proposition}

As in \S\ref{sec:width/height}, let $H_1$ and $H_2$ be the two horoballs that are incident to $H$ along its two extreme edges.
By construction, the non-backtracking path $\sigma$ from $H_1$ to $H_2$ is contained in $H$; we indicate the vertices $\sigma$ passes through in order by $p_1,\ldots,p_w$.
See \autoref{fig:branches}.

\begin{figure}[h]
\centering
\includegraphics[width=10.5cm]{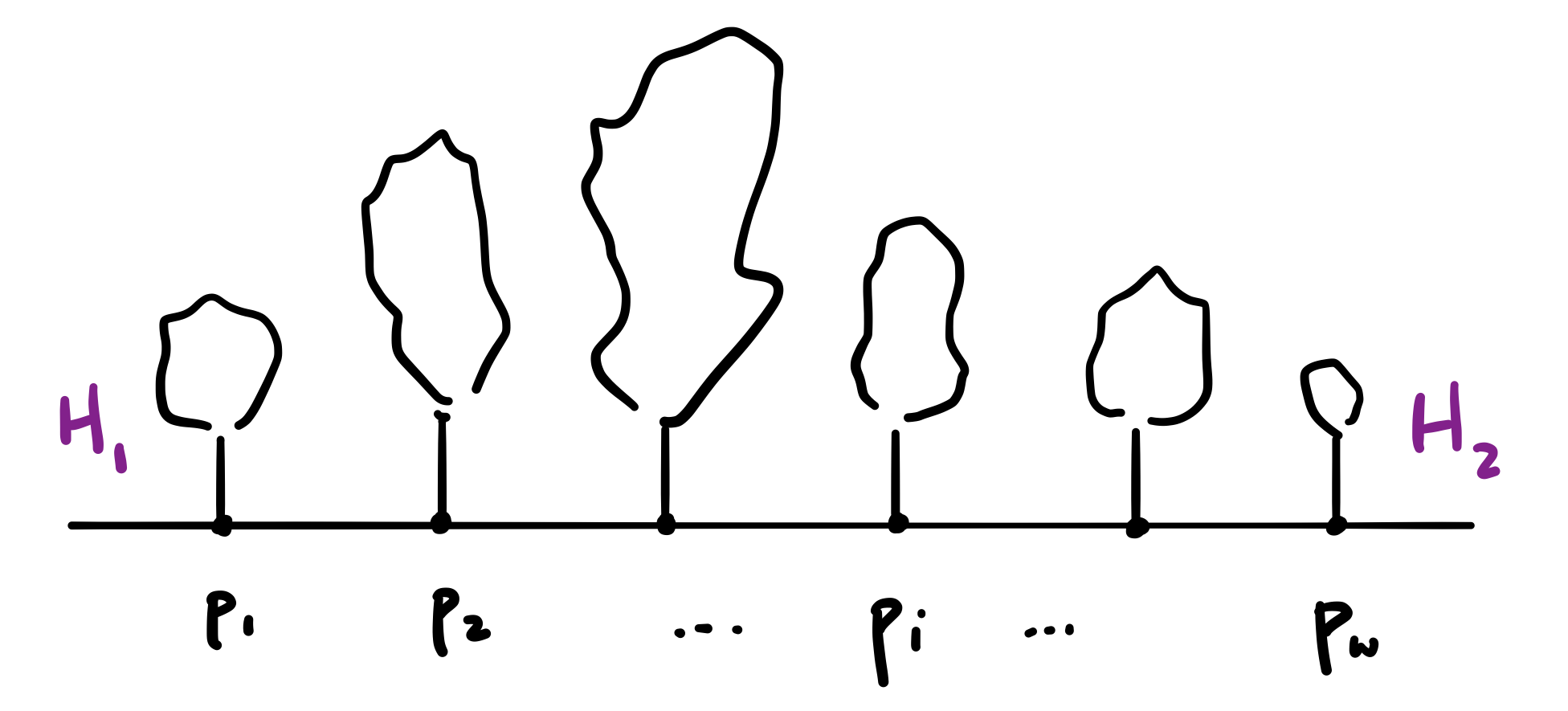}
\caption{The horoball $H$ determines branches for $\tau$, and extreme horoballs $H_1$ and $H_2$.}
\label{fig:branches}
\end{figure}

For each $j$, the complement $\tau \setminus p_j$ consists of three components, and we indicate (the closure of) the unique such component that doesn't intersect $H$ as the $j$th \emph{branch} $B(j)$.

\begin{figure}
\centering
\includegraphics[width=10cm]{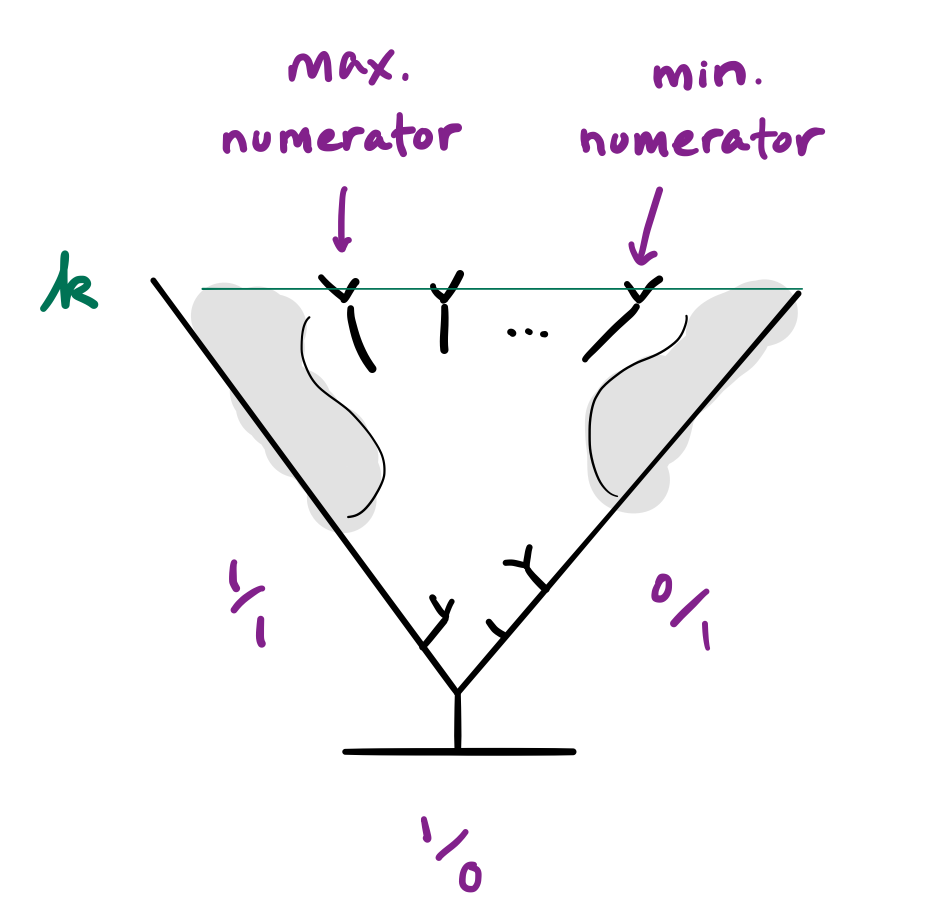}
\caption{Each branch has minimal and maximal numerators at height $k$.}
\label{pic:numerators}
\end{figure}

Label $H$ with $1/0$, label the horoball neighboring $H$ along the edge $\overline{p_j p_{j+1}}$ with $0/1$, and label the horoball neighbor of $H$ along $\overline{p_{j-1}p_j}$ with $1/1$. Now Farey addition determines how to fill in labels for the remaining horoballs, and let $\mathrm{ht}_H \, B(j)$ indicate the maximum denominator among Farey labels for horoballs intersecting $B(j)$. 
The vertices of $B(j)$ at height $k$ are given by $\frac {a_1}k, \frac{a_2}k,\ldots,\frac{a_{n_k}}k$. The \emph{minimum (resp.~max) numerator at height $k$} of $B(j)$, relative to $H$, is $\min a_i$ (resp.~$\max a_i$). See \autoref{pic:numerators}.

\begin{figure}
\centering
\includegraphics[width=10cm]{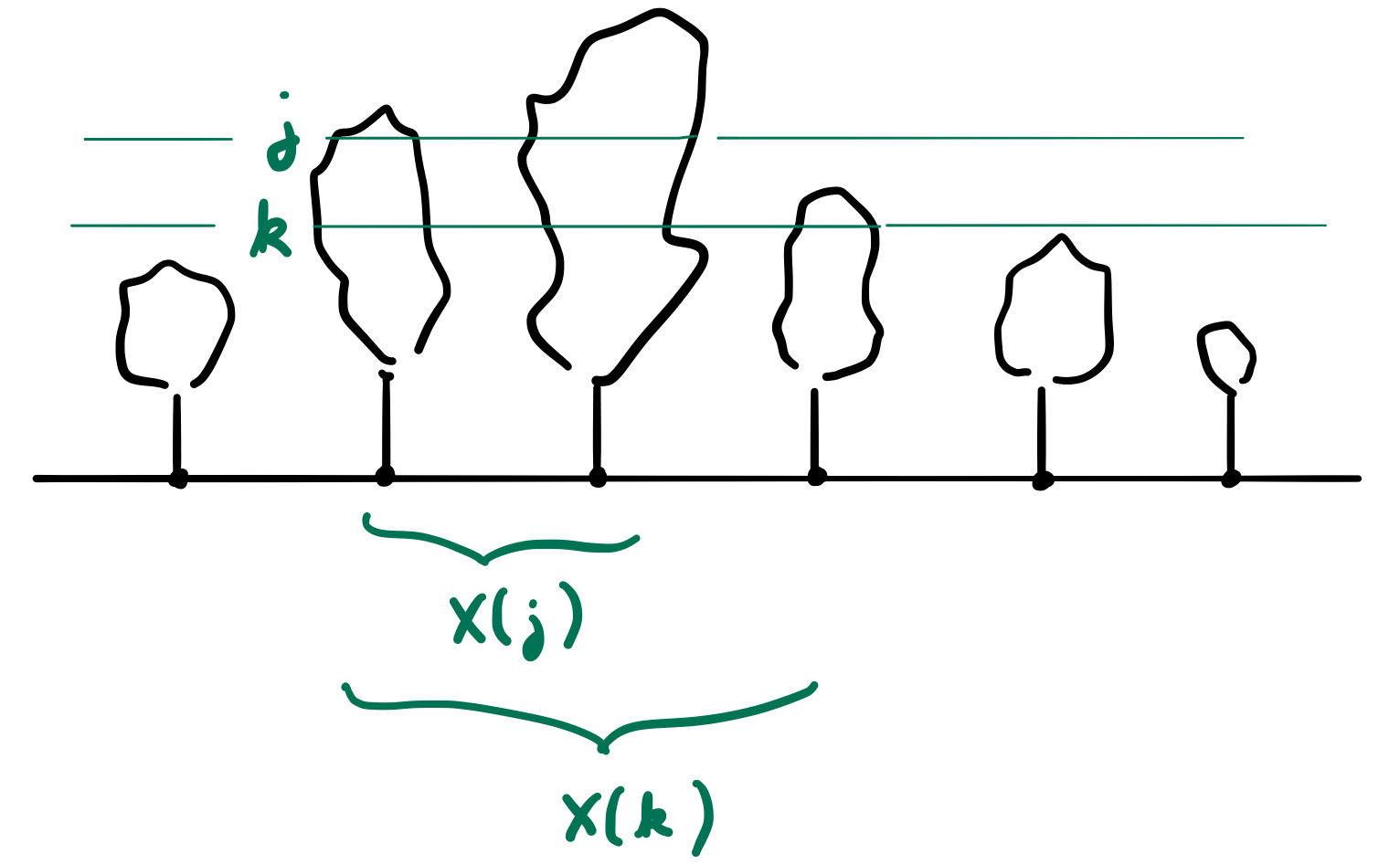}
\caption{The horoballs of $\tau$ may be filtered according to heights from $H$.}
\label{fig:heights}
\end{figure}

Observe that we may count the $n$ horoball regions of $\tau$ by filtering them according to their heights on the branches.
That is, for each $i\in \bN$, let $X(k) = \{ j : \mathrm{ht}_H \, B(j) \ge k\}$ (that is, the set of indices $j$ where the $j$th branch has height at least $k$), and let $x_k = \#X(k)$.
See Figure~\ref{fig:heights}. 

\begin{figure}
\centering
\includegraphics[width=10cm]{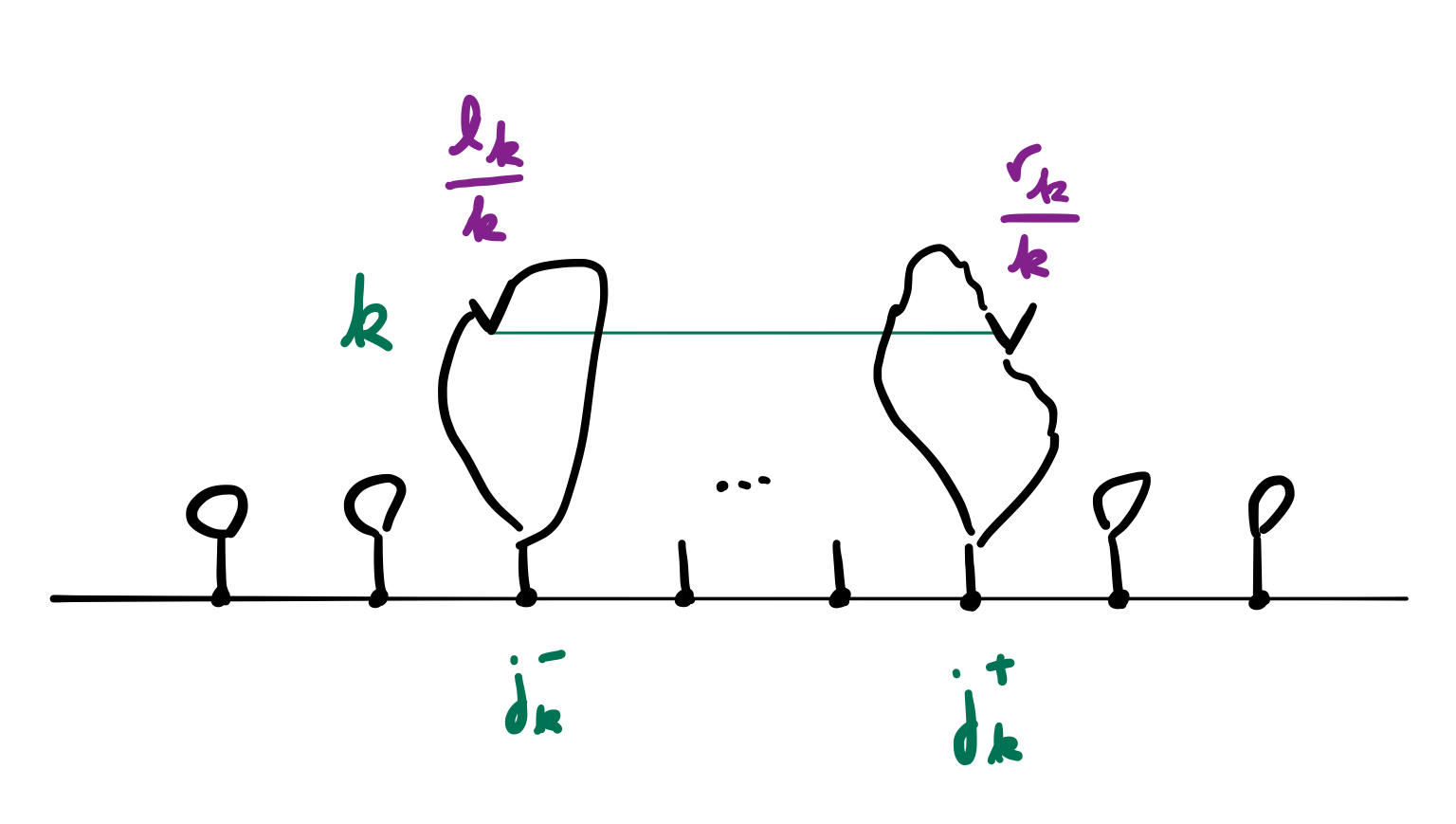}
\caption{Extremal horoballs of $\tau$ at height $k$ from $H$.}
\label{fig:function}
\end{figure}

Consider $j_k^+$ (resp.~$j_k^-$), the maximal (resp.~minimal) index of $X(k)$. Let $\ell_k$ be the maximal numerator of $B(j_k^-)$ and let $r_k$ be the minimal numerator of $B(j_k^+)$.
See \autoref{fig:function}.

\begin{lemma}
\label{lem:count n}
We have 
\[
\sum_{k=1}^h \phi(k)x_k \ge n + \sum_{k=1}^h \phi(k) 
+ \sum_{k=1}^h \frac{\phi(k)}k (r_k - \ell_k)
- O\left( h \log h\right)~.
\]
\end{lemma}

\begin{proof}
Observe that the sum of the number of horoballs of $\tau$ at height $k$ relative to $H$, as $k$ goes from $1$ to $h$, is exactly $n-2$.

The horoballs of $B(j)$ at height $k$ relative to $H$ have size at most $\phi(k)$, so the total number of horoballs of $\tau$ at height $k$ from $H$ is at most $\phi(k) x_k$. Observe that we may count the vertices of $B(j_k^-)$ and $B(j_k^+)$ with slightly more care: When the maximal and minimal numerators of $B(j)$ at height $k$ are $\ell$ and $r$, the horoballs of $B(j)$ at height $k$ relative to $H$ are at most $\phi(k)-\#[r-1]_k$, and at most $\#[\ell]_k$.
Therefore the total number of horoballs of $\tau$ at height $k$ relative to $H$ are at most 
\[
\phi(k)x_k - (\underbrace{\phi(k)-\#[\ell_k]_k}_{\text{overcount in }B(j_k^-)} ) - \underbrace{\#[r_k-1]_k}_{\text{overcount in }B(j_k^+)} ~.
\]
By \eqref{eq:euler count}, the latter is at most 
\[
\phi(k)x_k   - \phi(k) + \frac{\ell_k}k \phi(k) - \frac{r_k-1}k \phi(k) + O(d(k))~.
\]
The sum of this expression as $k$ goes from $1$ to $h$ is at least $n-2$, so rearranging we find that
\[
\sum_{k=1}^h \phi(k)x_k \ge n +\sum_{k=1}^h \phi(k) + \sum_{k=1}^h \frac{\phi(k)}k (r_k-\ell_k) - 2- \sum_{k=1}^h \frac{\phi(k)}k - C\sum_{k=1}^h d(k)~.
\]
By \eqref{eq:dirichlet} and \eqref{eq:walfisz}, the last three terms can be replaced by $O(h\log h)$.
\end{proof}

The reader may observe how \autoref{lem:count n} is somewhat suggestive of \eqref{eq:goal}.

Given heights $k$ and $k'$, consider the intersection $I_{kk'}$ between the horoball of $B(j_k^-)$ with maximal numerator $\ell_k$ and the horoball of $B(j_{k'}^+)$ of minimal numerator $r_{k'}$.
We may compute:
\begin{equation*}
I_{kk'} = kk' 
|j_{k'}^+ - j_k^-|
+ k'\ell_k - kr_{k'}~
\end{equation*}
Therefore, we have
\begin{equation*}
I_{kk'}+I_{k'k} = kk'\left( \big| j_{k'}^+-j_k^-\big| + \big| j_{k}^+-j_{k'}^-\big| \right) + (k'\ell_k - kr_{k'} + k\ell_{k'}-k'r_k)~.
\end{equation*}
Notice that $\big| j_{k'}^+-j_k^-\big| + \big| j_{k}^+-j_{k'}^-\big| \ge x_k+x_{k'}-2$, so dividing by $kk'$ we find
\begin{equation}
\label{eq:sum of Is}
\frac1{kk'} \left(I_{kk'} + I_{k'k}\right) \ge x_k+x_{k'} - 2 + \left(\frac{\ell_k}k- \frac{r_k}k\right) + \left(\frac{\ell_{k'}}{k'}  - \frac{r_{k'}}{k'}\right)~.
\end{equation}

With \autoref{lem:count n} in mind, we would like to choose pairs $\{k,k'\}\subset [h]$ so that the sum over the choices made of the terms `$x_k+x_{k'}$' on the righthand side of \eqref{eq:sum of Is} is equal to $\sum \phi(k)x_k$.
The following proposition makes this idea feasible.

\begin{proposition}
\label{prop:Gammas}
For each $h\in \bN$, there is a graph $\Gamma_h$ satisfying:
\begin{enumerate}
\item 
The vertex set of $\Gamma_h$ is given by $\{1,\ldots, h\}$.
\item \label{item:valence}
The valence of vertex $k$ is $\phi(k)$.
\item \label{item:weights}
The sum $\displaystyle \sum_{k\sim k'} \frac2{kk'}$ over the edges of $\Gamma_h$ is equal to $1$.
\end{enumerate}
\end{proposition}

See \autoref{pic:graph} for a picture of $\Gamma_6$.

\begin{proof}
Declare $k\sim k'$ when $\gcd(k,k')=1$ and $k+k'>h$.

For property~\eqref{item:valence}, choose a vertex $k$. 
Each integer $i$ relatively prime to $k$ can be shifted by $k$ to $i+k$, another integer relatively prime to $k$. 
For $1\le i\le k$, we may choose the maximum $n$ so that $i+nk\le h$. 
The result is a bijection of $[k]_k$, the integers in $[k]$ relatively prime to $k$, with the set of integers $k'$, relatively prime to $k$, less than $h$, and so that $k+k'>h$. Therefore the valence of $k$ is $\#[k]_k=\phi(k)$.

For property~\eqref{item:weights}, observe that, to transform $\Gamma_h$ into $\Gamma_{h+1}$, the edges $k\sim k'$ with $k+k'=h+1$ are deleted and replaced by edges $k\sim(h+1)$ and $k'\sim(h+1)$. Because $k+k'=h+1$, the edge weight $\frac2{kk'}$ in $\Gamma_h$ is equal to the sum of edge weights $\frac2{k(h+1)}+\frac2{k'(h+1)}$ in $\Gamma_{h+1}$, so the sum $\sum_{k\sim k'}\frac2{kk'}$ is independent of $h$.  
For the base case $h=2$, observe that $\frac2{1\cdot2}=1$.
\end{proof}

\begin{figure}
\centering
\includegraphics[height=7cm]{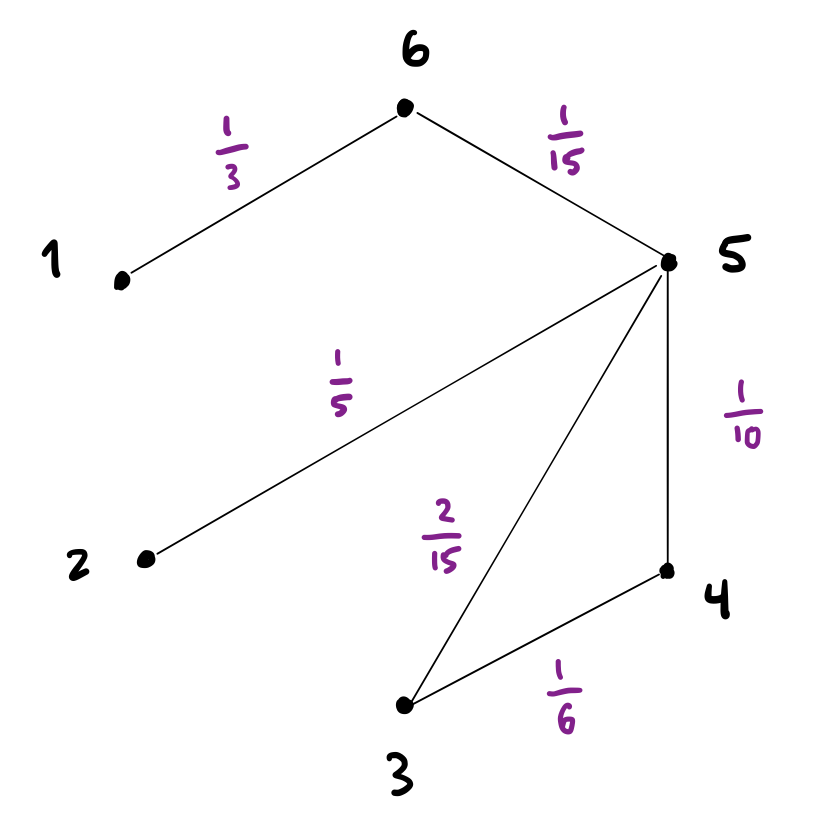}
\caption{The graph $\Gamma_6$ as described in \autoref{prop:Gammas}.
The sum of the edge weights is $1$.}
\label{pic:graph}
\end{figure}

Let $E$ indicate the set of edges of $\Gamma_h$. Now \autoref{prop:Gammas} and \eqref{eq:sum of Is} imply:

\begin{align*}
\sum_E \frac1{kk'} \left(I_{kk'} + I_{k'k}\right)
&\ge \ 
\sum_E  \ \left( x_k+x_{k'} - 2 + \left(\frac{\ell_k}k- \frac{r_k}k\right) + \left(\frac{\ell_{k'}}{k'}  - \frac{r_{k'}}{k'}\right) \right)  \\
&= \sum_k \phi(k) x_k \ -\  \sum_k \phi(k) \ + \ \sum_k \frac{\phi(k)}k \left( \ell_k  - r_k \right) 
\end{align*}
Applying \autoref{lem:count n}, we find that
\begin{equation}
\label{eq:convex estimate}
\sum_E \frac1{kk'} \left(I_{kk'} + I_{k'k}\right) \ge n - O\left( h \log h\right)~.
\end{equation}
Because $\sum_E \frac2{kk'} =1$ by \autoref{prop:Gammas}, inequality \eqref{eq:convex estimate} proves \autoref{prop: control with height}.

\section{Finding a horoball of controlled relative height}
\label{sec:controlling height}

For many $\tau\in\cA_n$, there exist $H$ so that the $\mathrm{ht}(\tau,H)$ is $O(1)$, so \autoref{prop: control with height} demonstrates that $\kappa(\tau)=n-O(1)$.
However, such a horoball need not exist, e.g.~every horoball of $\ach(n)$ has height at least $\approx (\frac{1+\sqrt 5}2)^{n/2}$.
Nonetheless, $\kappa(\ach(n))$ is quite large (on the order $(\frac{1+\sqrt 5}2)^n$), so one might hope that it is always possible to find horoballs of small height relative to $\kappa(\tau)$. 
We show:

\begin{proposition} 
\label{prop: exists horoball}
There exists a constant $C>0$ so that, for any $\tau\in \cA_n$, there exists a horoball $H$ of $\tau$ so that the height of $\tau$ relative to $H$ is controlled as $\mathrm{ht}(\tau,H) \le C \sqrt{\kappa(\tau)}$.
\end{proposition}

\begin{remark}
\label{rem: sharpness heights}
The reader can observe that the conclusion above fits the data in \autoref{table}. 
It is tempting to hope for an improvement of \autoref{prop: exists horoball} along the following lines: 
as $\kappa(\tau)$ gets closer to $\min \kappa$ (e.g.~if $\kappa(\tau)\le n$), one should be able to find horoballs of $\tau$ with relative heights $\ll \sqrt n$.

On the other hand, the row containing $\tau=\Far(h)$, with $n\approx \frac3{\pi^2}h^2$, makes this hope seem quite remote. 
Indeed, $\kappa(\Far(h))$ is greater than $n$ only by the innocuous looking linear factor $\frac {\pi^2}3 \approx 3.3$, and yet every horoball has relative height $\ge h \approx \sqrt n$.
\end{remark}

\begin{proof}
Let $K_1,K_2$ be horoballs of $\tau$ so that $\iota(K_1,K_2)=\kappa(\tau)$, and
let $r=\log \kappa(\tau)$.
By \autoref{lem:i vs d}, we have $d_{\bH^2}(K_1,K_2) = 2r$.

Let $x\in \bH^2$ be the midpoint of the geodesic from $K_1$ to $K_2$.
By \autoref{lem:distance from horoball} there is some Farey horoball $H$ so that $d_{\bH^2}(H,x)\le \log \frac2{\sqrt3}$. 
The Farey horoball $H$ is incident to the geodesic between $K_1$ and $K_2$, so by convexity it must be a horoball in $\tau$.

\begin{figure}
\centering
\includegraphics[height=10cm]{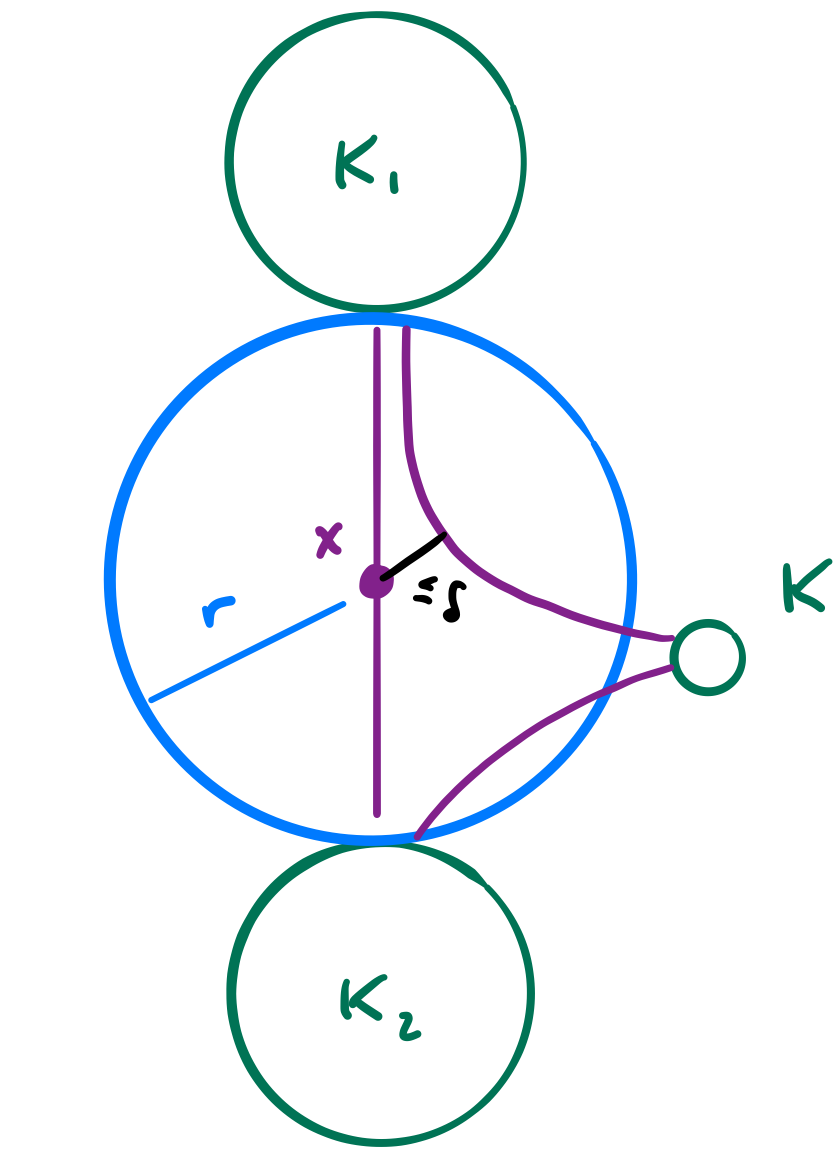}
\caption{Hyperbolicity guarantees that a horoball $K$ far from $x$ is far from some $K_i$ as well.}
\label{fig:pic-thin-triangle}
\end{figure}

We claim that $H$ satisfies the requisite bound. 
Let $K$ be any other horoball of $\tau$. Because $\bH^2$ is $\delta$-hyperbolic, the point $x$ is within $\delta$ of the geodesic segment between $K$ and $K_i$ for $i$ equal to either $1$ or $2$.
A standard application of the triangle inequality (see \autoref{fig:pic-thin-triangle}) then yields 
\[
d_{\bH^2}(K,K_i) \ge d_{\bH^2}(K,x) + d_{\bH^2}(x,K_i) - 2\delta = d_{\bH^2}(K,x) + r -2\delta~.
\] 
Because $d_{\bH^2}(K,K_i)\le 2r$, we conclude that $d_{\bH^2}(K,x)\le r+2\delta$, and
\[
d_{\bH^2}(H,K)\le d_{\bH^2}(H,x) + d_{\bH^2}(x,K) \le \log \frac2{\sqrt 3}+  r +2\delta~.
\]
By \autoref{lem:i vs d} we conclude that
\[
\iota(H,K) = e^{\frac12 d_{\bH^2}(H,K) } 
\le \frac2{\sqrt 3}e^{2\delta} \sqrt{\kappa(\tau)}~. \qedhere
\]
\end{proof}

\begin{remark}
\label{rem: Agol proof}
Ian Agol has suggested a slightly different version of the above proof: choose Farey labels for the horoballs in $\tau$, and enlarge the horoballs by $\log \kappa(\tau)+\log\frac2{\sqrt3}$. 
A variation on \autoref{lem:distance from horoball} together with \autoref{lem:i vs d} implies that every trio of these horoballs mutually intersect, so by Helly's theorem there is a point $x$ in their common intersection \cite{Helly}, and one may finish as above.
\end{remark}

\section{From kappa to eta}
\label{sec:pf of main thm}

As stated in the introduction, it is an exercise to show that 
\begin{equation*}
\label{eq:eta kappa}
\eta_T(k) = \max \{ n: \kappa_T(n)\le k\}~.
\end{equation*}
By \autoref{thm: kappa}, we may conclude $\eta_T(k)\le \max\{n: n-C\sqrt n\log n \le k\}$.

\autoref{main thm} now follows from the following lemma:

\begin{lemma}
\label{lem:functions}
Suppose that $C>0$ is a constant, and that $f:\bR\to\bR$ is an increasing, sublinear function, with $f(x)=o(x)$.
There is a $D>0$ so that for any $k\ge 1$ we have $\max\{ x: x- Cf(x)\le k\} \le k+ D f(k)$.
\end{lemma}

\begin{proof}
Because $f(x)=o(x)$, there is a $C_1>0$ large enough so that $C_1-1>Cf(C_1)$. 
Because $f$ is sublinear, for any $k\ge 1$ we have
\[
(C_1-1)k > Ck \, f(C_1) \ge C \, f(C_1k)~.
\]
Adding $k$ to both sides and rearranging we find
\begin{equation}
\label{eq:too big}
C_1k-C\, f(C_1k) > k~.
\end{equation}

Let $F(k)=\max\{ x: x-Cf(x)\le k\}$.
By \eqref{eq:too big} we have $F(k) < C_1k$. 
Of course, by definition of $F(k)$ we have $F(k) - Cf(F(k)) \le k$. 
Because $f$ is increasing and sublinear, we find
\[
F(k) \le k+C f(F(k)) \le k+C f(C_1k) \le k+CC_1 f(k)~,
\]
as claimed.
\end{proof}

Because $\sqrt x\log x$ is increasing and sublinear, this completes the proof of \autoref{main thm}.

\begin{remark}
\label{rem:sloppy lemma}
The conclusion of \autoref{lem:functions} holds under much weaker assumptions. For instance, sublinearity of $f$ can be replaced by the assumption that there is some $C_2>0$ so that $f(x+y) $ is at most $C_2 f(x) + C_2 f(y)$.
\end{remark}

\bigskip

\bibliographystyle{alpha}
\newcommand{\etalchar}[1]{$^{#1}$}

\end{document}